\documentclass{amsart}

\usepackage[english]{babel}
\usepackage[T1]{fontenc}
\usepackage{leftidx}
\usepackage[numbers]{natbib}
\usepackage{bigints}
\usepackage[new]{old-arrows}
\usepackage{enumitem}
\usepackage{empheq}
\usepackage{xcolor}
\usepackage{array}
\usepackage{graphicx}
\usepackage{lipsum}
\usepackage{amssymb}
\usepackage{amsthm,mathtools}
\usepackage{needspace}
\usepackage[makeroom]{cancel}

\usepackage{geometry}
\newtheorem{theorem}{Theorem}[section]
\newtheorem{lema}[theorem]{Lemma}
\newtheorem{question-non}[]{}

\newtheorem{cor}[theorem]{Corollary}
\newtheorem{observation}[theorem]{Remark}
\newtheorem{example}[theorem]{Example}

\title{Immersion of gradient almost Yamabe solitons into warped product manifolds}

\author{Tokura, W. $^{1}$}
\address{$^{1}$ Instituto Federal Goiano, 75380-000, Av. Wilton Monteiro da Rocha, s/n, Trindade, GO, Brazil.}
\email{williamisaotokura@hotmail.com $^{1}$}

\author{Adriano, L. $^{2}$}
\address{$^{2}$ Universidade Federal de Goi\'as, IME, 131, 74001-970, Goi\^ania, GO, Brazil.}
\email{levi@ufg.br $^{2}$}

\author{Batista, E. $^{3}$}
\address{$^{3}$ Universidade Federal de Goi\'as, IME, 131, 74001-970, Goi\^ania, GO, Brazil.}
\email{edbatista@gmail.com.br $^{3}$}

\author{Bezerra, A. C. $^{4}$}
\address{$^{4}$ Instituto Federal Goiano, 75380-000, Av. Wilton Monteiro da Rocha, s/n, Trindade, GO, Brazil.}
\email{adriano.bezerra@ifgoiano.edu.br $^{4}$}

\thanks{$^{1,3}$ Supported by CAPES}



\begin{document}

\begin{abstract}The purpose of this article is to study the geometry of gradient almost Yamabe solitons immersed into warped product manifolds $I\times_{f}M^{n}$ whose potential is given by the height function from the immersion. First, we present some geometric rigidity on compact solitons due to a curvature condition on the warped product manifold. In the sequel, we investigate conditions for the existence of totally geodesic, totally umbilical and minimal solitons. Furthermore, in the scope of constant angle immersions, a classification of rotational gradient almost Yamabe soliton immersed into $\mathbb{R}\times_{f}\mathbb{R}^{n}$ is also made.
\end{abstract}
\maketitle
\section{Introduction}
\label{intro}

The concept of gradient almost Yamabe soliton, introduced in the celebrated work \cite{barbosa2013conformal}, corresponds to a natural generalization of gradient Yamabe solitons \cite{hamilton1988ricci} and Yamabe metrics \cite{yamabe1960deformation}. We recall that a Riemannian manifold $(\Sigma^n, g)$ is an almost Yamabe soliton if it admits a vector field $X\in \mathfrak{X}(\Sigma)$ and a smooth function $ \lambda: \Sigma \rightarrow \mathbb{R}$ satisfying the equation 
\begin{equation}\label{def1}
    \frac{1}{2}\mathfrak{L}_{X}g=(scal_{g}-\lambda)g,
\end{equation}
where $\mathfrak{L}_{X}g$ and $scal_{g}$ stand, respectively, for the Lie derivative of $g$ in the direction of $X$ and the scalar curvature of $g$. The quadruple $(\Sigma^{n}, g, X, \lambda)$ is classified into three types according to the sign of
$\lambda$: expanding if $\lambda < 0$, steady if $\lambda = 0$ and shrinking if $\lambda > 0$. If $\lambda$ occurs as a constant, the soliton is usually referred to as an \textit{Almost Yamabe soliton}.
It may happen that $X=\nabla h$ is the gradient field of a smooth real function $h$ on $M$, called \textit{potential}, in which case the soliton $(\Sigma^{n}, g, h, \lambda)$ is referred to as a \textit{gradient almost Yamabe soliton}. Equation \eqref{def1} then becomes
\begin{equation}\label{def2} 
Hess_{g}h=(scal_{g}-\lambda)g,
\end{equation}
where $Hess_{g}h$ is the Hessian of $h$. We pointed out that, if $h$ is constant, the soliton is called \textit{trivial}.

Almost Yamabe solitons are special solutions of Hamilton’s Yamabe flow \cite{hamilton1988ricci} and can be viewed as stationary points of the Yamabe flow in the space of Riemannian metrics on $\Sigma^n$ modulo diffeomorphisms and scalings of $\Sigma^n$. In this way, the study of the analytical and geometrical properties of almost Yamabe solitons becomes essential for understanding the behavior of the Yamabe flow.



Recently, much efforts have been devoted to understanding the geometry of almost Yamabe solitons and almost Ricci solitons, both isometrically immersed into space forms \cite{aquino2017characterizations, barros2011immersion,chen2014classification, chen2018yamabe, cunha2018r, seko2019classification}. Chen and Deshmukh \cite{chen2018yamabe}, study Yamabe solitons whose soliton field is the tangent component from the position vector on Euclidean space and, as a result, they given some rigidity results.  Under a concurrent vector field assumption on the soliton field, Seko and Maeta \cite{seko2019classification} showed that any almost Yamabe soliton has a gradient almost Yamabe soliton structure. Furthermore, for almost Yamabe solitons on ambient spaces furnished with a concurrent vector field, the authors give a classification of such solitons. On the other hand, Aquino et al. \cite{aquino2017characterizations} present the study of gradient almost Ricci solitons immersed into the space of constant sectional curvature $M^{n+1}\subset \mathbb{R}^{n+2}$ with potential given by height function from the soliton associated to a fixed direction on $\mathbb{R}^{n+2}$.

The above works bring us to light that important geometric results are obtained if we choose some appropriate soliton vector field. In this sense, a vector field that already proven themselves to be a rich source for produce examples of solitons fields is the one generated by the gradient of the height function from the immersion. Examples of solitons in which the height function is taken as the potential function are given in \cite{aquino2017characterizations,  barbosa2013conformal,barros2011immersion, barros2014characterizations, chen2014classification,cunha2018r, tokura2018warped}.

From the previous works, gradient almost Yamabe solitons in which height function is chosen as the potential might be interesting for further investigation. Moreover, to extend the above works to a larger class of ambient spaces, it appears convenient to consider the immersions into a sufficiently large family of manifolds, which includes the spaces of constant sectional curvature. A natural metric, which includes in its range the spaces of constant sectional curvature, is described by warped product metrics \cite{o1983semi}. Warped product manifolds have already proven themselves to be a profitable ambient space to obtain a wide range of distinct geometrical proprieties for immersions (cf. \cite{alias2007constant, caminha2009complete, colares2012some, de2019characterizations, dillen2011classification}). In this context, as in \cite{alias2007constant}, we can extend the concept of height function from the immersion by means of the projection onto the base of the warped product (see Section \ref{Pre}).



The purpose of this manuscript is to study the geometry of gradient almost Yamabe solitons $(\Sigma^{n}, g, h, \lambda)$ immersed into warped product manifold $I\times_{f}M^{n}$ whose potential $h$ is given by the height function from the immersion. In this setting, we derive a necessary and sufficient condition for the immersion be a gradient almost Yamabe soliton. We use this result to investigate conditions for the existence of totally geodesic, totally umbilical, minimal and trivial solitons. Furthermore, when the ambient space is taking as $\mathbb{R}\times_{f}\mathbb{R}^n$, we provide the classification of rotational gradient almost Yamabe soliton with a constant angle.

This manuscript is organized in the following way: In Section \ref{Pre}, we recall some basic facts and notations that will appear throughout the paper. Afterward, in Section \ref{Eexamples}, we exhibit some examples of immersions satisfying the gradient almost Yamabe soliton equation \eqref{def2}, and we establish our first main results concerning the geometry of these geometric objects. Finally, in Section \ref{classification}, we provide the classification of rotationally symmetric gradient almost Yamabe solitons.
\section{Preliminaries}
\label{Pre}

Let $M^{n}$ be a connected, $n$-dimensional oriented Riemannian manifold, $I\subset\mathbb{R}$ an
interval and $f : I \rightarrow (0,\infty)$ a smooth function. In the product differentiable
manifold $\overline{M}^{n+1}= I \times M^{n}$, consider the projections $\pi_{I}$ and $\pi_{M}$ onto the spaces $I$ and $M$, respectively. A particular class of Riemannian manifold is the one obtained by furnishing $\overline{M}^{n+1}$ with
the metric 
\[\langle \ ,\ \rangle=\pi_{I}^{\ast}(dt^2)+f^2(\pi_{I})\pi_{M}^{\ast}(g_{M}),\] 
such a space is called a warped product manifold with base $I$, fiber $M^n$ and warping function $f$. In this setting, for a fixed $t_{0}\in\mathbb{R}$, we say that $\Sigma_{t_{0}}^{n}:=\{t_{0}\}\times M^{n}$ is a slice of $\overline{M}^{n+1}$. 

Let $\overline{\nabla}$ and $\nabla$ the Levi-Civita connection in $I\times_{f}M^{n}$ and $\Sigma^n$, respectively. Then, the Gauss-Weingarten formulas for a isometric immersion $\psi:\Sigma^{n}\rightarrow I\times_{f}M^{n}$ are given by
\begin{equation}\label{eq1}\overline{\nabla}_{X}Y=\nabla_{X}Y+\langle AX,Y\rangle N,\qquad AX=-\overline{\nabla}_{X}N,
\end{equation}
for any $X\in \mathfrak{X}(\Sigma)$, where $A:T\Sigma^n\rightarrow T\Sigma^n$ denotes the Weingarten operator of $\Sigma^n$ with respect to its Gauss map $N$. In this scope, we consider two particular functions naturally attached to such a hypersurface $\Sigma^n$, namely, the height function $h:=(\pi_{I})|_{\Sigma}$ and the angle function $\theta=\langle N,\partial_{t}\rangle$, where $\partial_{t}$ is the standard unit vector field tangent to $I$. By a straightforward computation we obtain that the gradient of $\pi_{I}$ on $I\times_{f}M^{n}$ is given by
\[\overline{\nabla}\pi_{I}=\langle \overline{\nabla}\pi_{I}, \partial_{t}\rangle \partial_{t}=\partial_{t},\]
so that the gradient of $h$ on $\Sigma^{n}$ is
\begin{equation}\label{eq2}\nabla h=(\overline{\nabla}\pi_{I})^{\top}=\partial_{t}^{\top}=\partial_{t}-\theta N,
\end{equation}
where $(\hspace{0,1cm}\cdot\hspace{0,1cm})^{\top}$ denotes the tangential component of a vector field
in $\mathfrak{X}(\overline{M})$ along $\Sigma^{n}$. In particular, we get
\[|\nabla h|^{2}=1-\theta^2,\]
where $|\cdot|$ denotes the norm of a vector field on $\Sigma^{n}$. 

Let $\overline{R}$ and $R$ be the curvature tensors of $I\times_{f}M^n$ and $\Sigma^n$, respectively. Therefore, for any $X$, $Y$, $Z\in \mathfrak{X}(\Sigma)$ we have the following Gauss equation:
\begin{equation}\label{eq3}
R(X,Y)Z=(\overline{R}(X,Y)Z)^{\top}+ \langle AX,Z\rangle AY- \langle AY, Z\rangle AX.
\end{equation}

Denote by $Ric$ the ricci tensor of $\Sigma^n$ and consider a local orthonormal frame $\{E_{i}\}_{i=1}^n$ of $\mathfrak{X}(\Sigma)$, as well as $X\in\mathfrak{X}(\Sigma)$. Then, it follows from the Gauss equation \eqref{eq3} that
\begin{equation}\label{Ric}
Ric(X,X)=\sum_{i=1}^{n}\langle\overline{R}(X,E_{i})X, E_{i}\rangle+nH\langle AX,X\rangle-\langle AX,AX\rangle.
\end{equation}
Moreover, taking into account the properties of the Riemannian tensor $\overline{R}$ of a warped product (see for instance Proposition 7.42 in \cite{o1983semi}), we deduce
\begin{align*}
\overline{R}(X,Y)Z=R^{M}(X^{\ast}, Y^{\ast})Z^{\ast}&-[(\log f)'(h)]^{2}\left[\langle X,Z\rangle Y-\langle Y,Z\rangle X\right]\\
&+(\log f)''(h)\langle Z,\partial_{t}\rangle\left[\langle Y,\partial_{t}\rangle X-\langle X,\partial_{t}\rangle Y\right]\\
&-(\log f)''(h)\left[\langle Y,\partial_{t}\rangle\langle X,Z\rangle-\langle X,\partial_{t}\rangle\langle Y,Z\rangle\right]\partial_{t},
\end{align*}
where $R^{M}$ is the curvature tensor of the fiber and $X^{\ast}=X-\langle X,\partial_{t}\rangle \partial_{t}$, $E_{i}^{\ast}=E_{i}-\langle E_{i},\partial_{t}\rangle \partial_{t}$ are, respectively, the projections of the tangent vector fields $X$ and $E_{i}$ onto $M^{n}$. Thus, we obtain that
\begin{equation}\label{ricc2}
\begin{split}
\sum_{i=1}^{n}\langle\overline{R}(X,E_{i})X,E_{i}\rangle=&f(h)^{-2}\sum_{i=1}^{n}K^{M}(X^{\ast}, E_{i}^{\ast})\Big{[}|X|^{2}-\langle X,\nabla h\rangle ^{2}-|X|^{2}\langle \nabla h, E_{i}\rangle ^{2}\\&-\langle X,E_{i}\rangle ^{2}+ 2\langle X, \nabla h\rangle\langle X, E_{i}\rangle\langle\nabla h, E_{i}\rangle\Big{]}+[(\log f)'(h)]^{2}\Big{(}|\nabla h|^{2}\\
&-(n-1)\Big{)}|X|^{2}-(n-2)(\log f)''(h)\langle X,\nabla h\rangle ^{2}-\frac{f''}{f}|\nabla h|^{2}|X|^{2},
\end{split}
\end{equation}
where $K^{M}$ is the sectional curvature of $M^{n}$, and hence from \eqref{Ric}, the scalar curvature of $\Sigma^n$ takes the following form
\begin{equation}\label{eq4}
\begin{split}
    scal_{g}=f(h)^{-2}&\sum_{i,j=1}^{n}K^{M}(E_{j}^{\ast}, E_{i}^{\ast})\Big{[}1-\langle E_{j},\nabla h\rangle ^{2}-\langle \nabla h, E_{i}\rangle ^{2}-\langle E_{j},E_{i}\rangle ^{2}\\
&+ 2\langle E_{j}, \nabla h\rangle\langle E_{j}, E_{i}\rangle\langle\nabla h, E_{i}\rangle\Big{]}+n[(\log f)'(h)]^{2}\left(|\nabla h|^{2}-(n-1)\right)\\
&-(n-2)(\log f)''(h)|\nabla h|^{2}-n\frac{f''}{f}|\nabla h|^{2}+n^{2}H^{2}-|A|^{2}.
\end{split}
\end{equation}


From \cite{o1983semi}, we know that $\overline{M}^{n+1}$ has constant sectional curvature $c$ if, and only if, $M^n$ has constant sectional curvature $k$ and the warping function $f$ satisfy the following ODE:
\begin{equation}\label{ssaa}
    \frac{(f')^2-k}{f^2}=-c=\frac{f''}{f}.
\end{equation}
We remark that a Riemannian manifold of constant sectional curvature $c\in \{-1,0,1\}$ can be expressed as a warped product manifold $I\times_{f}M^{n}$, namely
  \begin{alignat*}{3}
    &\mathbb{R}^{n+1}\setminus\{0\}=(0,+\infty)\times_{f}\mathbb{S}^n  && \hspace{4cm}\textup{with}\hspace{0.1cm}f(t)=t,\\
    &\mathbb{R}^{n+1}=\mathbb{R}\times_{f}\mathbb{R}^{n}   &&
    \hspace{4cm}
    \textup{with}\hspace{0.1cm}f(t)=1,\\
     &\mathbb{S}^{n+1}\setminus\{\pm p\}=(0,\pi)\times_{f}\mathbb{S}^{n}   &&
    \hspace{4cm}
    \textup{with}\hspace{0.1cm}f(t)=\sin t,\\
    &\mathbb{H}^{n+1}=\mathbb{R}\times_{f}\mathbb{R}^{n}     &&
    \hspace{4cm}
    \textup{with}\hspace{0.1cm}f(t)=e^t,\\
    &\mathbb{H}^{n+1}\setminus\{p\}=(0,+\infty)\times_{f}\mathbb{S}^{n}       &&
    \hspace{4cm}
    \textup{with}\hspace{0.1cm}f(t)=\sinh t.
  \end{alignat*}
After a straightforward calculation, we easily see that these warped product models trivially verify \eqref{ssaa}.

Proceeding, in order to establish our main results, we will need the following key lemma, which provides a necessary and sufficient condition to a hypersurface be a gradient almost Yamabe soliton with height function as the potential. 


\begin{lema}\label{prop}Let $\psi:\Sigma^{n}\rightarrow I\times_{f}M^{n}$ be an isometric immersion. Then $(\Sigma^{n},g)$ is a gradient almost Yamabe soliton with potential $h=(\pi_{I})|_{\Sigma}$ if, and only if, 
\begin{equation}\label{trace}
    (scal_{g}-\lambda)g(X,Y)=(\log f)'(h)\left[g(X,Y)-dh\otimes dh(X,Y)\right]+\theta g(AX,Y)
\end{equation}
for all $X$, $Y\in\mathfrak{X}(\Sigma)$.
\end{lema}

\begin{proof}Taking into account the properties of the Levi-Civita connection of a warped product (see, for instance, Proposition 7.35 in \cite{o1983semi}), it easily follows that
\[\overline{\nabla}_{X}\partial_{t}=\frac{f'}{f}(X-\langle X,\partial_{t}\rangle\partial_{t}),\qquad \forall X\in \mathfrak{X}(\Sigma).\]
Thus, from equations \eqref{eq1} and \eqref{eq2}, we deduce the following expression for the Hessian of $h$
\begin{equation*}
Hess(h)(X)=\nabla_{X}\nabla h=\frac{f'(h)}{f(h)}\left(X-\langle X,\nabla h\rangle \nabla h\right)+ \langle N, \partial_{t} \rangle AX,
\end{equation*}
therefore, 
\begin{equation*}
Hessh(X,Y)=g(\nabla_{X}\nabla h,Y)=\frac{f'(h)}{f(h)}\left[g(X,Y)-dh\otimes dh(X,Y)\right]+\theta g(AX,Y).
\end{equation*}
The result follows by the fundamental equation \eqref{def2}.
\end{proof}




We finalize this section by quoting the generalized Hopf's maximum principle due to S.T. Yau. In the following, $L^{1}(\Sigma)$ stands for the space of the Lebesgue integrable functions on $\Sigma^n$.

\begin{lema}\label{yau}\textup{(\cite{yau1976some})} Let $(\Sigma^n,g)$ be a complete, noncompact Riemannian manifold. If $h:\Sigma\rightarrow\mathbb{R}$ is a smooth subharmonic function such that $|\nabla h|\in L^{1}(\Sigma)$, then $h$ must be actually harmonic.
\end{lema}

\section{Examples and main results}
\label{Eexamples}
Before present the main results, we will exhibit some examples of immersions satisfying the gradient almost Yamabe soliton equation 
\eqref{def2}.

\begin{example}
Let $(\mathbb{S}^{n},g_{1})$ the standard sphere immersed into Euclidean space $(\mathbb{R}^{n+1},g_{0})$. According to \cite{barbosa2013conformal}, if we taken the height function from the sphere given by
\[h:\mathbb{S}^n\to\mathbb{R},\quad x\mapsto g_1(x,\eta_{1}),\]
where $\eta_{1}=(1,0,\dots,0)\in\mathbb{R}^{n+1}$ and $x=(x_{1},\dots,x_{n+1})\in\mathbb{S}^{n}$ is the position vector, then $(\mathbb{S}^{n},g_{1})$ is a gradient almost Yamabe soliton with height function as the potential and soliton function given by $\lambda=\frac{1}{n}(\Delta h-scal_{g_{1}})$.
\end{example}

\begin{example}
Let $\mathbb{P}^{n}:=\{(x_{1},x_{2},x_{3},\dots,x_{n+1})\in\mathbb{R}^{n+1}\hspace{0,1cm}|\hspace{0,1cm}x_{2}=0\}$ the hyperplane isometrically imersed into Euclidean space $(\mathbb{R}^{n+1},g_{0})$. Hence, taking the height function from the hyperplane given by
\[h:\mathbb{P}^n\to\mathbb{R},\quad x\mapsto g_{0}( x,\eta_{1}),\]
where $\eta_{1}=(1,0,\dots,0)\in\mathbb{R}^{n+1}$ and $x=(x_{1},0 ,x_{3},\dots,x_{n+1})\in\mathbb{P}^{n}$, we deduce that $(\mathbb{P}^{n}, g_{0})$ is a steady gradient almost Yamabe soliton with height function as the potential function.
\end{example}

\begin{example}Consider the hyperbolic space $\mathbb{R}\times_{e^{t}}\mathbb{R}^{n}$ furnished with the warped product structure. It is
well known that the horospheres of the hyperbolic space are totally umbilical hypersurfaces isometric to $\mathbb{R}^{n}$ and correspond to slices ${\{t_{0}\}}\times \mathbb{R}^{n}$, $t_{0}\in \mathbb{R}$. Hence, taking the inclusion $i:\{t_{0}\}\times\mathbb{R}^{n}\rightarrow \mathbb{R}\times\mathbb{R}^{n}$, we deduce that the height function satisfies $h(x)=t_{0}$, and then the standard Euclidean space $\{t_{0}\}\times\mathbb{R}^{n}$ is a trivial gradient almost Yamabe soliton with potential $h(x)=t_{0}$.
\end{example}


The above example allows us to conclude, in a broad sense, that for each fixed number $t_{0}\in I$, the inclusion $i:\{t_{0}\}\times M^{n}\rightarrow I\times M^{n}$ produces a constant height function $h(x)=t_{0}$. Hence, $\{t_{0}\}\times M^{n}$ is a trivial gradient almost Yamabe soliton with potential $h(x)=t_{0}$. This observation allows us to produce infinitely many examples of gradient almost Yamabe solitons immersions, i.e.,


\begin{example}Every manifold $\Sigma^n\subset M^{n}$, isometrically included into the warped product manifold $I\times_{f} M^n$ is a trivial gradient almost Yamabe soliton with potential $h=(\pi_{I})|_{\Sigma}=const.$ and scalar curvature $scal_{g}=\lambda$.
\end{example}

The next example deals with a rotationally symmetric gradient almost Yamabe soliton with a constant angle.

\begin{example}\label{Example}
Let  $\psi:\Sigma^2=(0,\infty)\times (0,2\pi)\rightarrow \mathbb{R}\times_{e^t}\mathbb{R}^2$ be an isometric immersion given by:
$$\psi(u,v)=(u\sqrt{1-\theta^2},-\frac{\theta}{\sqrt{1-\theta^2}} e^{-u\sqrt{1-\theta^2}}\cos v,-\frac{\theta}{\sqrt{1-\theta^2}} e^{-u\sqrt{1-\theta^2}}\sin v),\quad \theta\in (0,1),$$
then, $\Sigma^2$ is a gradient almost Yamabe soliton with potential $h(u,v)=u\sqrt{1-\theta^2}$ and soliton function $\lambda=scal_{g}$ (see Section \ref{classification}).
\end{example}

\begin{figure}[ht]
\begin{center}
\advance\leftskip-3cm
\advance\rightskip-3cm
\includegraphics[keepaspectratio=true,scale=0.33]{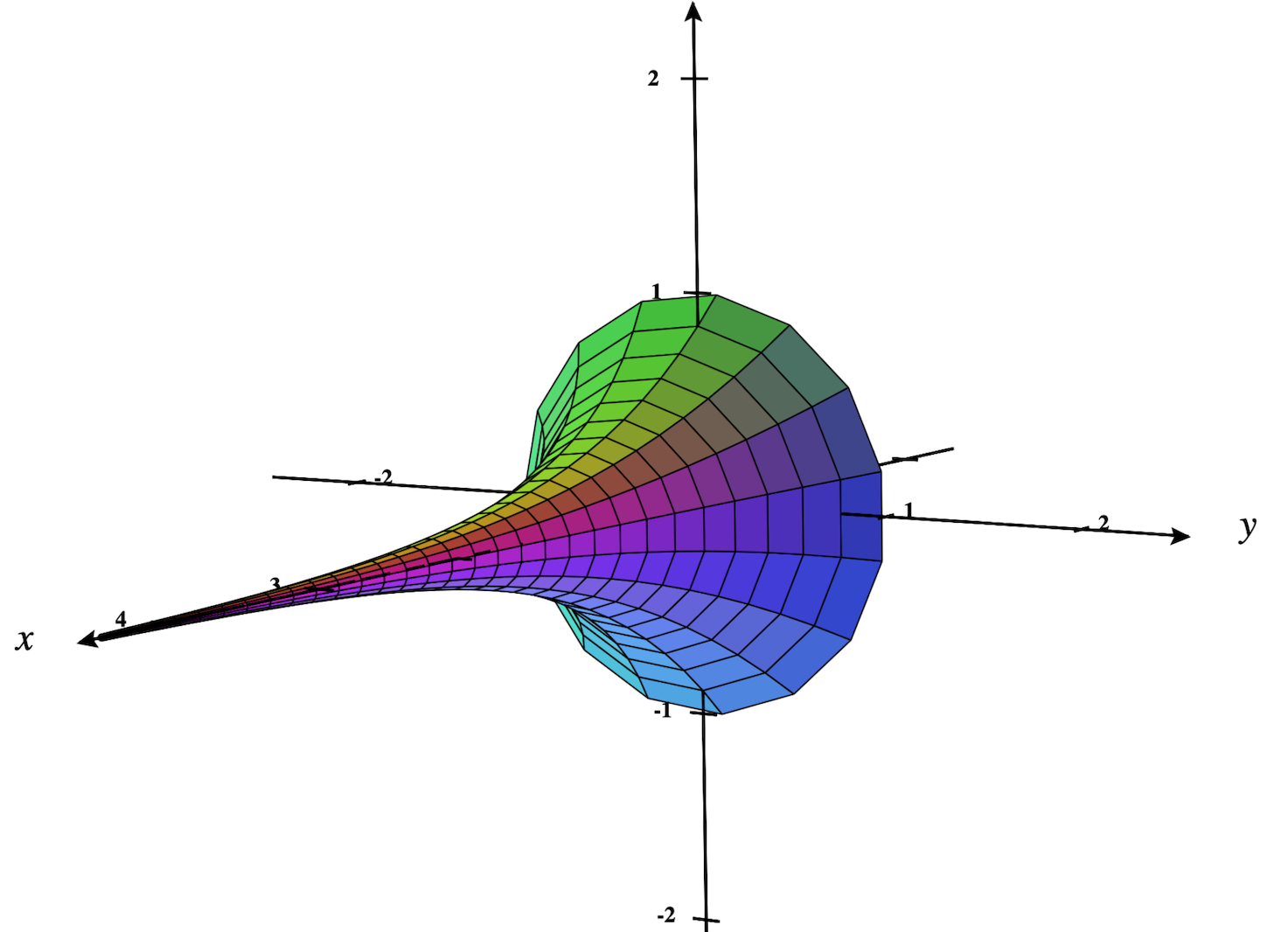}
\caption{Rotational soliton immersed into hyperbolic space with $\theta=\frac{\sqrt{2}}{2}$.}
\end{center}\end{figure}

Initially, we focus our attention on compact gradient almost Yamabe soliton immersions $\psi:\Sigma^{n}\rightarrow I\times_{f} M^{n}$. It has been known that every compact gradient Yamabe soliton
is of constant scalar curvature, hence, trivial since $h$ is
harmonic, see \cite{daskalopoulos2013classification, hsu2012note}. For gradient almost Yamabe solitons, the previous result was generalized by Barbosa et al. \cite{catino2012global}, where the authors proved that any compact gradient almost Yamabe soliton $(M^{n},g, h,\lambda)$ satisfying
\begin{equation}\label{int}
    \int_{M}g(\nabla\lambda, \nabla h)dv_{g}\geq0,
\end{equation}
is trivial. Our first result disregards \eqref{int} in favor of a hypothesis about the geometry of $I\times_{f}M^{n}$ and produces the following result.




\begin{theorem}\label{Te1}Let $(\Sigma^{n},g,h,\lambda)$ be a compact gradient almost Yamabe soliton immersed into $I\times_{f} M^{n}$ whose fiber $M^n$ has sectional curvature $k_{M}\geq\sup_{\substack{I}}((f')^{2}-ff'')$ and the warping function $f$ satisfies:
\begin{equation}\label{hyp}\frac{f''(h)}{f(h)}\leq \dfrac{n+1}{n^2}H^2,\qquad\quad 0\leq |\theta|^{-1}(\log f)'(h)\leq H,
\end{equation}
then the soltion is trivial, i.e., the potential $h$ is constant. \end{theorem}

\begin{proof}Using the gradient almost Yamabe soliton condition \eqref{def2}, we know that
\begin{equation*}
    \Delta h=n(scal_{g}-\lambda).
\end{equation*}
Therefore, from the Kazdan-Warner condition \cite{bourguignon1987scalar}(see Theorem II.9),
\begin{equation}\label{kazdan}
    \int_{\Sigma}g(\nabla scal_{g},\nabla h)dv_{g}=0,
\end{equation}
we obtain that
\begin{equation}\label{Eq1}
\begin{split}
\int_{\Sigma}(scal_{g}-\lambda)^{2}dv_{g}&=\frac{1}{n}\int_{\Sigma}(scal_{g}-\lambda)\Delta h dv_{g}\\
&=\frac{1}{n}\int_{\Sigma}g(\nabla scal_{g},\nabla h)dv_{g}-\frac{1}{n}\int_{\Sigma}g(\nabla\lambda,\nabla h)dv_{g}\\
&=-\frac{1}{n}\int_{\Sigma}g(\nabla\lambda,\nabla h)dv_{g}.
\end{split}
\end{equation}

Next, from the structural equations of gradient almost Yamabe solitons, presents in \cite{barbosa2013conformal}(see Lemma 2.3), we get 
\[Ric(\nabla h)+(n-1)\nabla(scal_{g}-\lambda)=0.\]
Hence,
\begin{equation}\label{Eq2}
\begin{split}
\int_{\Sigma}Ric(\nabla h,\nabla h)dv_{g}&=-(n-1)\int_{\Sigma}g(\nabla(scal_{g}-\lambda),\nabla h) dv_{g}\\
&=-(n-1)\int_{\Sigma}g(\nabla scal_{g},\nabla h)dv_{g}+(n-1)\int_{\Sigma}g(\nabla\lambda,\nabla h)dv_{g}\\
&=(n-1)\int_{\Sigma}g(\nabla\lambda,\nabla h)dv_{g},
\end{split}
\end{equation}
where in the last equality we use again the Kazdan-Warner condition \eqref{kazdan}. Combining equations \eqref{Eq1} and \eqref{Eq2}, we produce
\begin{equation}\label{scalt}
\int_{\Sigma}(scal_{g}-\lambda)^{2}dv_{g}=-\frac{1}{n(n-1)}\int_{\Sigma}Ric(\nabla h,\nabla h)dv_{g}.
\end{equation}

Now, consider a local orthonormal frame $\{E_{i}\}_{i=1}^n$ of $\mathfrak{X}(\Sigma)$. Then, it follows from the Ricci expression \eqref{Ric} that
\begin{equation}\label{riccc}
\begin{split}
Ric(\nabla h,\nabla h)&=\sum_{i=1}^{n}\langle\overline{R}(\nabla h,E_{i})\nabla h, E_{i}\rangle+nH\langle A(\nabla h),\nabla h\rangle-\langle A(\nabla h),A(\nabla h)\rangle.
\end{split}
\end{equation}
Observe that from \eqref{ricc2} we shall estimate
\begin{equation}\label{efes}
\begin{split}
\sum_{i=1}^{n}\langle\overline{R}(\nabla h,E_{i})\nabla h,E_{i}\rangle\geq&-\log f''(h)\sum_{i=1}^{n}\Big{[}|\nabla h|^{2}-|\nabla h|^{4}-|\nabla h|^{2}\langle \nabla h, E_{i}\rangle ^{2}-\langle \nabla h,E_{i}\rangle ^{2}\\&+ 2|\nabla h|^{2}\langle \nabla h, E_{i}\rangle^{2}\Big{]}+[(\log f)'(h)]^{2}\Big{(}|\nabla h|^{2}-(n-1)\Big{)}|\nabla h|^{2}\\&-(n-2)(\log f)''(h)|\nabla h|^{4}-\frac{f''(h)}{f(h)}|\nabla h|^{4}\\
\geq& -\frac{f''(h)}{f(h)}(n-1)|\nabla h|^2.
\end{split}
\end{equation}
On the other hand, from equation \eqref{trace} of Lemma \ref{prop}, we deduce that

\begin{equation}\label{qq1}
   nH\langle A(\nabla h),\nabla h\rangle= \Bigg{(}(n-1)H\frac{|\nabla h |^2}{\theta}\frac{f'(h)}{f(h)}+nH^2\Bigg{)}|\nabla h|^2,
   \end{equation}
as well as,
\begin{equation}\label{qq2}
\begin{split}\langle A\nabla h,A\nabla h\rangle&=\Bigg{(}\frac{n-1}{n}\frac{|\nabla h|^2}{\theta}\frac{f'(h)}{f(h)}+H\Bigg{)}\langle\nabla h,A\nabla h\rangle\\
&= \Bigg{(}\frac{n-1}{n}\frac{|\nabla h|^2}{\theta}\frac{f'(h)}{f(h)}+H\Bigg{)}^2|\nabla h|^2\\
    &= \left(\frac{(n-1)^2}{n^2}\frac{|\nabla h|^4}{\theta^2}\left(\frac{f'(h)}{f(h)}\right)^2+2\frac{n-1}{n}\frac{|\nabla h|^2}{\theta}\frac{f'(h)}{f(h)}H+H^2\right)|\nabla h|^2.
    \end{split}
\end{equation}
Replacing \eqref{efes}, \eqref{qq1} and \eqref{qq2} into \eqref{riccc}, we arrived at the following
\begin{equation*}
\begin{split}
    Ric(\nabla h,\nabla h)\geq\Bigg{[}-\frac{f''(h)}{f(h)}(n-1)+\frac{n-2}{n}(n-1)H\frac{|\nabla h |^2}{\theta}\frac{f'(h)}{f(h)}\hspace{2.8cm}\\
 -\frac{(n-1)^2}{n^2}\frac{|\nabla h|^4}{\theta^2}\left(\frac{f'(h)}{f(h)}\right)^2+(n-1)H^2\Bigg{]}|\nabla h|^2,
    \end{split}
\end{equation*}
and taking account hypothesis \eqref{hyp}, we obtain that
\begin{equation*}
\begin{split}
    Ric(\nabla h,\nabla h)\geq(n-1)\Bigg{[}-\frac{f''(h)}{f(h)}-\frac{n-2}{n}H^2
    -\frac{n-1}{n^2}H^2
    +H^2\Bigg{]}|\nabla h|^2\geq0.
    \end{split}
\end{equation*}
Thus, from equation \eqref{scalt} we conclude that $\Sigma^{n}$ is a trivial gradient almost Yamabe soliton.
\end{proof}


If we utilize the space forms model as warped product rather than the more commonly used model of such spaces, we produce from Theorem \ref{Te1} the following result.





\begin{cor}Let $(\Sigma^{n},g,h,\lambda)$ be a compact gradient almost Yamabe soliton immersed into $I\times_{f}M^n$. Then, the following statements hold:
\begin{itemize}[leftmargin=16pt,align=left,labelwidth=\parindent,labelsep=0pt]
\item [\textup{(a)}] \hspace{0,2cm}If $I\times_{f}M^n$ is the Euclidean sphere $(0,\pi)\times _{\sin t}\mathbb{S}^n$ and the mean curvature of $\Sigma^n$ satisfies:
\[ 0\leq |\theta|^{-1}\cot{(h)}\leq H, \]
then $\Sigma^n$ is trivial.
\item [\textup{(b)}] \hspace{0,2cm}If $I\times_{f}M^n$ is the Euclidean space $(0,+\infty)\times _{t}\mathbb{S}^n$ and the mean curvature of $\Sigma^n$ satisfies:
\[ 0\leq |\theta|^{-1}h^{-1}\leq H, \]
then $\Sigma^n$ is trivial.
\item [\textup{(c)}] \hspace{0,2cm}If $I\times_{f}M^n$ is the hyperbolic space $(0,+\infty)\times _{\sinh{t}}\mathbb{S}^n$ and the mean curvature of $\Sigma^n$ satisfies:
\[ \frac{n^2}{n+1}\leq H^2,\qquad 0\leq |\theta|^{-1}\coth{(h)}\leq H, \]
then $\Sigma^n$ is trivial.
\end{itemize}
\end{cor}


Next, we focus our attention on hypersurfaces immersed into a particular class of warped product manifolds, which holds when the ambient space is an arbitrary Riemannian product manifold.

\begin{theorem}Let $(\Sigma^{n},g,h,\lambda)$ be a gradient almost Yamabe soliton immersed into a Riemannian product $I\times M^{n}$. If the angle function $\theta$ does not change sign, then $\Sigma^{n}$ is a totally umbilical hypersurface of $I\times M$.
\end{theorem}
\begin{proof}
First, let us consider a local orthonormal frame $\{E_{i}\}_{i=1}^n$ of $\mathfrak{X}(\Sigma)$ associated with the Weingarten operator, i.e., $A(E_{i})=\lambda_{i}E_{i}$, where $\{\lambda_{i}\}_{i=1}^{n}$ are the principal curvatures of $\Sigma^n$. Since the warping function $f$ is constant, we deduce from Lemma \ref{prop} that 
\[(scal_{g}-\lambda)\delta_{ij}=\theta\lambda_{i}\delta_{ij},\]
which implies that 
\begin{equation*}\lambda_{i}=\theta^{-1}(scal_{g}-\lambda),\qquad 1\leq i\leq n.
\end{equation*}
Therefore, $\Sigma^n$ is totally umbilical with mean curvature $H=\theta^{-1}(scal_{g}-\lambda)$.
\end{proof}

In the particular case in which the ambient space is the Euclidean space, we obtain the following classification. 

\begin{cor}Let $(\Sigma^{n},g,h,\lambda)$ be a gradient almost Yamabe soliton immersed into the Euclidean space $\mathbb{R}^{n+1}$. If $\theta$ does not change sign, then $\Sigma^n$ is a hyperplane or a hypersphere of $\mathbb{R}^{n+1}$.
\end{cor}

In order to investigate minimal gradient almost Yamabe solitons, we prove the following.

\begin{theorem}\label{mini}Let $(\Sigma^{n},g,h,\lambda)$ be a minimal gradient almost Yamabe soliton immersed into $I\times_{f}M^n$ with $f'(h)\geq0$, then the scalar curvature of $\Sigma^{n}$ satisfies $scal_g\geq\lambda$. Moreover, if $h$ reaches the maximum, then $scal_{g}\equiv\lambda$, $f'(h)=0$ and $\Sigma^n$ is a slice of $I\times M^{n}$.
\end{theorem}
\begin{proof}
Since $(\Sigma^{n},g,h,\lambda)$ is a minimal gradient almost Yamabe soliton, we deduce from the trace of \eqref{trace} in Lemma \ref{prop} that
\begin{equation*}
\begin{split}
    (scal_{g}-\lambda)n&=\frac{f'(h)}{f(h)}\left(n-|\nabla h|^{2}\right).
\end{split}
\end{equation*}
From $|\nabla h|^{2}=1-\theta^2$ and by the hypothesis $f'(h)\geq0$, we derive 
\begin{equation}\label{label}
    (scal_{g}-\lambda)n=\frac{f'(h)}{f(h)}\left(n-|\nabla h|^{2}\right)=\frac{f'(h)}{f(h)}\left(n-1+\theta^2\right)\geq0,
\end{equation}
which proves that $scal_{g}\geq\lambda$. On the other hand, it follows from equation \eqref{def2} that
\begin{equation}\label{maximo principio}
    \Delta h=n(scal_{g}-\lambda)\geq0.
\end{equation}
Now, assume that $h$ attains its maximum $h_{0}$ in the point $x_{0}\in \Sigma^{n}$ and define
\begin{equation*}
\Omega_{0}:=\{x\in \Sigma^{n}\hspace{0.2cm} ;\hspace{0.2cm} h(x)=h_{0}\}.
\end{equation*}
since $x_{0}\in \Omega_{0}$, it must be closed and non-empty. Let now $y\in \Omega_{0}$, then applying the maximum principle (see \cite{gilbarg2015elliptic} p. 35 ) to \eqref{maximo principio} we obtain that, $h(x)=h_{0}$ in a neighborhood of $y$ so that $\Omega_{0}$ is open. Connectedness of $\Sigma^{n}$ yields $\Omega_{0}=\Sigma$. Hence $h$ is constant, which implies that $\Sigma^{n}$ is a slice, $scal_{g}=\lambda$ and $f'(h)=0$.
\end{proof}

As a consequence of Theorem \ref{mini} we obtain a condition for nonexistence of minimal immersion of gradient almost Yamabe solitons into the hyperbolic space and Euclidean space. More precisely, we derive the following corollary.

\begin{cor}Let $\psi:\Sigma^n\rightarrow I\times_{f}M^n$ be an isometric immersion of a gradient almost Yamabe soliton $(\Sigma^{n},g,h,\lambda)$  into a warped product $I\times_{f}M^n$. Then the following conditions hold.
\begin{itemize}[leftmargin=16pt,align=left,labelwidth=\parindent,labelsep=0pt]
\item [\textup{(a)}] \hspace{0,2cm}If $I\times_{f}M^n=\mathbb{R}\times_{e^t}\mathbb{R}^n$ and $\lambda>-n(n-1)-|A|^2$, then $\psi$ can not be minimal.
\vspace{0,2cm}
\item [\textup{(b)}]\hspace{0,2cm}If $I\times_{f}M^n=(0,\infty)\times_{t}\mathbb{S}^{n}$ and $\lambda>-|A|^2$, then $\psi$ can not be minimal.
\end{itemize}
\end{cor}

As another application of Theorem \ref{mini} we also get

\begin{cor}\label{coro}Let $(\Sigma^{n},g,h,\lambda)$ be a minimal gradient almost Yamabe soliton immersed into the Riemannian product manifold $I\times M^n$. If $h$ reaches it's maximum, then $(\Sigma^n,g)$ is isometric to $M^n$.
\end{cor}

\begin{observation}
Corollary \ref{coro}, reveals that there not exist a compact minimal gradient almost Yamabe soliton $\Sigma^n$ immersed into the Riemannian product manifold $I\times M^n$ with $M^n$ noncompact.
\end{observation}

Our next result provides an extension of Theorem 1.5 in \cite{barros2011immersion} in the scope of gradient almost Yamabe solitons immersed into warped product manifolds.


\begin{theorem}\label{Te111}Let $(\Sigma^{n},g,h,\lambda)$ be a gradient almost Yamabe soliton immersed into $I\times_{f} M^{n}$ whose fiber $M^n$ has sectional curvature $k_{M}\leq\inf_{\substack{I}}((f')^{2}-ff'')$.
\begin{itemize}[leftmargin=16pt,align=left,labelwidth=\parindent,labelsep=0pt]
\item [\textup{(a)}] \hspace{0,2cm}If $|\nabla h|\in L^{1}(\Sigma)$ and the soliton function satisfies
   $$\lambda\geq -n(n-1)\frac{f''(h)}{f(h)}+n^{2}H^{2},$$
then $\Sigma^n$ is totally geodesic, with scalar curvature $scal_{g}=-n(n-1)\frac{f''}{f}$ and $k_{M}=(f')^{2}-ff''$.
\item [\textup{(b)}] \hspace{0,2cm}If $|\nabla h|\in L^{1}(\Sigma)$ and the soliton function satisfies
     $$\lambda\geq n(n-1)\Big{(}H^{2}-\frac{f''(h)}{f(h)}\Big{)},$$ 
then $\Sigma^n$ is totally umbilical, with scalar curvature $scal_{g}=n(n-1) (H^2-\frac{f''}{f})$ and with $k_{M}=(f')^{2}-ff''$.
\end{itemize}
\end{theorem}

\begin{proof}First, note that our hypothesis under the sectional curvature $k_{M}$ jointly with \eqref{eq4} implies that
\begin{equation}\label{eq4444}
\begin{split}
    scal_{g}&\leq \frac{\inf_{I}((f')^2-ff'')}{f^2}(n-1)\big{(}n-2|\nabla h|^2\big{)}+n[(\log f)'(h)]^{2}\left(|\nabla h|^{2}-(n-1)\right)\\
&\quad-(n-2)(\log f)''(h)|\nabla h|^{2}-n\frac{f''}{f}|\nabla h|^{2}+n^{2}H^{2}-|A|^{2}\\
&\leq-(n-1)(\log f)''(h)\big{(}n-2|\nabla h|^2\big{)}+n[(\log f)'(h)]^{2}\left(|\nabla h|^{2}-(n-1)\right)\\
&\quad-(n-2)(\log f)''(h)|\nabla h|^{2}-n\frac{f''}{f}|\nabla h|^{2}+n^{2}H^{2}-|A|^{2}\\
&\leq -n(n-1)\frac{f''(h)}{f(h)}+n^{2}H^{2}-|A|^{2}.
\end{split}
\end{equation}
Hence, combining our assumption on soliton function $\lambda$ with inequality \eqref{eq4444}, we deduce that
\begin{equation}\label{deltinha}
\Delta h=n(scal_{g}-\lambda)=n\left(-n(n-1)\frac{f''(h)}{f(h)}+n^{2}H^{2}-\lambda-|A|^{2}\right)\leq 0.
\end{equation}

Now, from Lemma \ref{yau}, we derive that $h$ is harmonic, and then from \eqref{deltinha}, $\Sigma^n$ must be totally geodesic with $scal_{g}=\lambda=-n(n-1)\frac{f''}{f}$. On the other hand, from $scal_{g}=\lambda$, we get that $k_{M}=(f')^2-ff''$.

For the second assertion, note that the traceless second fundamental form of $\Sigma^n$, namely, $\Phi=A-HI$, satisfies $|\Phi|^{2}=tr(\Phi^2)=|A|^{2}-nH^{2}\geq0$ and equality holds if, and only if, $\Sigma^{n}$ is totally umbilical. In this direction, from the hypothesis on $\lambda$ and equation \eqref{deltinha}, it yields
\begin{equation}\label{max2}
    \begin{split}
        \Delta h=n(scal_{g}-\lambda)= n\left[n(n-1)\left(-\frac{f''(h)}{f(h)}+H^{2}\right)-\lambda-|\Phi|^{2}\right]\leq0.
    \end{split}
\end{equation}
Hence, again from Lemma \ref{yau}, we deduce that $scal_{g}=\lambda=n(n-1)\left(H^2-\frac{f''}{f}\right)$ and $|\Phi|^{2}=0$, which gives that $\Sigma^n$ is totally umbilical. On the other hand, from $scal_{g}=\lambda$, we get that $k_{M}=(f')^2-ff''$.
\end{proof}

Proceeding, it is a well-known fact that any compact gradient almost Yamabe soliton with constant scalar curvature is isometric to euclidean sphere $\mathbb{S}^n$ (cf.\cite{barbosa2013conformal}). Using this result, we derive the following rigidity result.


\begin{theorem}\label{Teo122}
Let $(\Sigma^{n},g,h,\lambda)$ be a compact gradient almost Yamabe soliton immersed into a space form $\overline{M}^{n+1}(c)$ of curvature $c$. If $\lambda\geq(n-1)c+n|H|^{2}$, then $(\Sigma^n,g)$ is isometric to Euclidean sphere $(\mathbb{S}^n,g_{1})$.
\end{theorem}
\begin{proof}Since a warped product $I\times_{f}M^n$ of constant curvature $c$ trivially fulfills the condition $k_{M}=\inf_{I}((f')^2-ff'')$, we obtain in a similar way as in the demonstration of Theorem \ref{Te111} that
\begin{equation}\label{ee}
\Delta h=n(scal_{g}-\lambda)=n\left(n(n-1)c+n^{2}H^{2}-\lambda-|A|^{2}\right)\leq 0,
\end{equation}
which implies that $-h$ is subharmonic, and from the maximum principle, $h$ must by constant. Hence, from \eqref{ee}, we derive that $|A|^2=0$ and  $scal_{g}=n(n-1)c$, so the result follows by Theorem 1.5 of \cite{barbosa2013conformal}.
\end{proof}

\begin{observation}
We remark that Theorem \ref{Te111} and Theorem \ref{Teo122} are obtained in the general case, without assumption $h=(\pi_{I})|_{\Sigma}$. However, upon assuming this condition, we obtain from item $(a)$ jointly with Lemma \ref{prop} that either $\Sigma^n$ is a slice, or $f$ is a constant and $\Sigma^n$ is a totally geodesic hypersurface into a product manifold of zero sectional curvature.

For item $(b)$, we deduce $Hessh=(scal_{g}-\lambda)g=0$, then either $h$ is constant, wich implies that $\Sigma^n$ is trivial, or $|\nabla h|\neq 0$ and $\Sigma^n$ splits along the gradient of $h$. In the last case, from Lemma \ref{prop}, we get
\[\frac{f'(h)}{f(h)}dh\otimes dh=\left(\frac{f'(h)}{f(h)}+\theta H\right)g.\]
Hence, from lemma 1 of \cite{barros2014characterizations}, we obtain $\log f(h)'|\nabla h|^2=0$,
which implies that $f$ is a constant. So, $\Sigma^n$ is a totally umbilical hypersurface into a Riemannian product manifold of zero sectional curvature. Finally, Theorem \ref{Teo122} in the particular case $h=(\pi_{I})|_{\Sigma}$  remains the same.
\end{observation}

\section{Classification of rotational gradient almost Yamabe solitons}
\label{classification}


In this section, we present a classification of rotational gradient almost Yamabe solitons immersed into $\mathbb{R}\times_{f}\mathbb{R}^{n}$ with potential $h:=(\pi_{I})|_{\Sigma}$ and constant angle $\theta\in (0,1)$. Following Dajczer and do Carmo \cite{do2012rotation}, we shall use the terminology of rotational hypersurface in $\mathbb{R}\times_{f}\mathbb{R}^{n}$ as a hypersurface invariant by the orthogonal group $O(n)$ seen as a subgroup of the isometries group of $\mathbb{R}\times_{f}\mathbb{R}^{n}$.



Initially, consider the coordinates $(t,x_{1},\dots,x_{n})$, as well as the standard orthonormal basis $\{\eta_{1}, \dots , \eta_{n+1}\}$ of $\mathbb{R}\times_{f}\mathbb{R}^{n}$. Then, up to isometry, we can assume the rotation axis to be $\eta_{1}$. Consider a parametrized by the arc length curve in the $tx_{n}$ plane given by
\begin{align*}
      \gamma \colon (t_{0}&,t_{1}) \longrightarrow \mathbb{R}\times_{f}\mathbb{R}^n\\
       &u \xmapsto{\hspace{0.5cm}}(\alpha(u),0,\dots,0,\beta(u)).
\end{align*}
Rotating this curve around the $t$-axis we obtain a \textit{rotational hypersurface} in $\mathbb{R}\times_{f}\mathbb{R}^{n}$. Now, in order to obtain a parametrization of a rotational hypersurface, consider the unit sphere $\mathbb{S}^{n-1}\subset\mathbb{R}^{n}=\textup{span}\{\eta_{2},\dots, \eta_{n+1}\}$ with orthogonal parametrization given by
\begin{align*}
&X_{1}=\cos v_{1},\quad X_{2}=\sin v_{1}\cos v_{2},\quad X_{3}=\sin v_{1}\sin v_{2}\cos v_{3}, \quad \dots\quad\\ 
&X_{n-1}=\sin v_{1}\sin v_{2}.\dots\sin v_{n-2}\cos v_{n-1}, X_{n}=\sin v_{1}\sin v_{2}\dots\sin v_{n-2}\sin v_{n-1}.
\end{align*}
Therefore, a parametrization of a rotational hypersurface $\Sigma^n$ with radial axis $\eta_{1}$ into $\mathbb{R}\times_{f}\mathbb{R}^n$ is give by
\begin{equation}\label{para}
\begin{split}
  \psi \colon &(t_{0},t_{1})\times (0,2\pi)^{n-1} \rightarrow \mathbb{R}\times_{f}\mathbb{R}^n\\[1ex]
  &(u,v_{1},\dots, v_{n-1}) \xmapsto{\hspace{0.3cm}} \alpha(u)\eta_{1}+\beta(u)X(v_{1},\dots,v_{n-1}),
  \end{split}
\end{equation}
where $$X(v_{1},\dots,v_{n-1})=(0,X_{1}(v_{1},\dots,v_{n-1}),\dots, X_{n}(v_{1},\dots,v_{n-1})).$$

In this setting, we provide the following classification.

\begin{theorem}
Let $\psi:\Sigma^{n}\rightarrow\mathbb{R}\times_{f}\mathbb{R}^{n}$ be a rotational hypersurface with constant angle $\theta\in(0,1)$. Then, up to constants, there exists a unique immersion which makes $\psi$ a rotational gradient almost Yamabe soliton which is given by
\begin{equation*}
\psi(u,v_{1},\dots, v_{n-1})=u\sqrt{1-\theta^2}\eta_{1}+\left(\frac{\theta}{1-\theta^2}\int^{u\sqrt{1-\theta^2}}\frac{ds}{f(s)}\right)X(v_{1},\dots,v_{n-1}),\quad f(t)=e^{t},
\end{equation*}
where $\eta_{1}=(1,0,\dots,0)\in\mathbb{R}^{n+1}$, $
-\infty<u<\infty$, $0< v_{1},\dots, v_{n-1}< 2\pi$ and $X$ is a sphere parametrization.
\end{theorem}

\begin{proof} Since $\psi:\Sigma^{n}\rightarrow\mathbb{R}\times_{f}\mathbb{R}^{n}$ is a rotational hypersurface, we deduce from \eqref{para} that 
\begin{equation}\label{tangente}
\begin{split}
    \psi_{u}&=\alpha'(u)\eta_{1}+\beta'(u)X,\\[1.0ex] \psi_{v_{i}}&=\beta(u)X_{v_{i}},\quad 1\leq i\leq n-1,
    \end{split}
\end{equation}
and then, the first fundamental form of $\Sigma^n$ takes the form


\begin{align}\label{I}
  I &=\begin{bmatrix}
    \quad 1 & 0 & \dots & 0\\[0.6em]
    \quad 0 & f(\alpha(u))^{2}\beta(u)^2 & \dots & 0 \\[0.6em]
    \quad \vdots & \vdots & \ddots & \vdots \\[0.6em]
    \quad 0 & 0 & \dots & f(\alpha(u))^{2}\beta(u)^{2}
  \end{bmatrix}.
\end{align}
The first fundamental equation \eqref{I} reveals that the induced metric on $\Sigma^n$ can be expressed by the warped product metric $g=du^2+\sigma(u)^{2}dv^2$ where $\sigma(u)=f(\alpha(u))\beta(u)$. In this case, it follows from the Levi-Civita connection on the warped product metric that:
\begin{equation}\label{xxx}
\begin{split}
&\nabla_{\psi_{u}}\psi_{u}=0,\\[1.5ex] &\nabla_{\psi_{u}}\psi_{v_{i}}=\nabla_{\psi_{v_{i}}}\psi_{u}=\frac{\sigma_{u}}{\sigma}\psi_{v_{i}},\\[1.5ex] &\nabla_{\psi_{v_{i}}}\psi_{v_{j}}=\psi_{v_{i} v_{j}}-\sigma\sigma_{u}\delta_{ij}\psi_{u}.
\end{split}
\end{equation}

From the tangent components \eqref{tangente}, we easily derive the following unit normal vector field for $\Sigma^n$
\begin{equation*}
    N=f(\alpha(u))\beta'(u)\eta_{1}-\frac{\alpha'(u)}{f(\alpha(u))}X(v_{1},\dots,v_{n-1}).
\end{equation*}
Hence, the hypersurface $\Sigma^n$ determines a constant angle hypersurface with constant angle $\theta$ if, and
only if,
\begin{equation}\label{unit2}\theta=\langle \partial_{t},N\rangle=f(\alpha(u))\beta'(u)=\text{constant}.
\end{equation}
Combining the unit condition for the rotational curve $\gamma(u)=(\alpha(u),0,\dots,0,\beta(u))$, i.e.,
\begin{equation*}
\alpha'(u)^{2}+f(\alpha(u))^{2}\beta'(u)^{2}=1,
\end{equation*}
and \eqref{unit2}, we deduce $\alpha'(u)=\sqrt{1-\theta^2}$, whose general solution is given by 
\begin{equation}\label{a}
\alpha(u)=u\sqrt{1-\theta^2}+c_{1},\quad c_{1}\in\mathbb{R}.
\end{equation}
Then, replacing equation \eqref{a} into \eqref{unit2} and solving in $u$, we derive the following expression
\begin{equation}\label{b}
    \beta(u)=\int^{u}\frac{\theta}{f(s\sqrt{1-\theta^2}+c_{1})}ds+c_{2}=\frac{\theta}{\sqrt{1-\theta^2}}\int^{u\sqrt{1-\theta^2}+c_{1}}\frac{ds}{f(s)}+c_{2},\quad c_{2}\in\mathbb{R}.
\end{equation}
Therefore, the rotational hypersurface takes the following form
\begin{equation}\label{sa}
\begin{split}
\psi=(u\sqrt{1-\theta^2}+c_{1})\eta_{1}+\left(\frac{\theta}{\sqrt{1-\theta^2}}\int^{u\sqrt{1-\theta^2}+c_{1}}\frac{ds}{f(s)}+c_{2}\right)X(v_{1},\dots,v_{n-1}).
\end{split}
\end{equation}

Now, in order to compute the Weingarten operator $A_{N}$, let us consider the following decomposition
\begin{equation}\label{deriva}\partial_{t}=\sqrt{1-\theta^2}\psi_{u}+\theta N.
\end{equation}
Taking the covariant derivative of \eqref{deriva} with respect $\psi_{v_{i}}$ and considering that the angle $\theta$ is constant, as well as the properties of the Levi-Civita connection of $\mathbb{R}\times_{f}\mathbb{R}^{n}$ (Proposition 7.35 in \cite{o1983semi}), we deduce that
\begin{equation}\label{xx}
\nabla_{\psi_{v_{i}}}\psi_{u}=\frac{\theta}{\sqrt{1-\theta^2}} A_{N}\psi_{v_{i}}+\frac{1}{\sqrt{1-\theta^2}}\frac{f'(\alpha(u))}{f(\alpha(u))}\psi_{v_{i}},\qquad \forall i\in\{1,\dots, n-1\}.
\end{equation}
Combining \eqref{xxx} and \eqref{xx}, yields

\begin{equation}\label{sigma}
    \sqrt{1-\theta^2}\frac{\sigma_{u}}{\sigma}\psi_{v_{i}}=\theta A_{N}\psi_{v_{i}}+\frac{f'(\alpha(u))}{f(\alpha(u))}\psi_{v_{i}},
\end{equation}
and therefore, from the expression of $\sigma$, we obtain that $\psi_{v_{i}}$ is an eigenvector for $A_{N}$ and satisfies

\begin{equation}\label{weing2}
   A_{N}\psi_{v_{i}}= \left(\frac{\sqrt{1-\theta^2}}{\sigma}-\frac{f'(\alpha(u))}{f(\alpha(u))}\theta\right)\psi_{v_{i}}.
\end{equation}
On the other hand, taking the covariant derivative of \eqref{deriva} with respect $X\in\mathfrak{X}(\Sigma)$ and using the Gauss-Weingarten formulas \eqref{eq1}, we deduce the following implications
\begin{equation*}
    \begin{split}
        \overline{\nabla}_{X}\partial_{t}&=\sqrt{1-\theta^2}\overline{\nabla}_{X}\psi_{u}+\theta \overline{\nabla}_{X}N\\
        &=\sqrt{1-\theta^2}\nabla_{X}\psi_{u}+\sqrt{1-\theta^2}g(A_{N}\psi_{u},X)N-\theta A_{N}X,
    \end{split}
\end{equation*}
and, again from the proprieties of the Levi-Civita connection of $\mathbb{R}\times_{f}\mathbb{R}^{n}$ \cite{o1983semi}, it follows
\begin{equation}\label{comparison}
\frac{f'(\alpha(u))}{f(\alpha(u))}\left(X-\sqrt{1-\theta^2}g(X,\psi_{u})\partial_{t}\right)=\sqrt{1-\theta^2}\nabla_{X}\psi_{u}+\sqrt{1-\theta^2}g(A_{N}\psi_{u},X)N-\theta A_{N}X.
\end{equation}
Comparing the tangent and the normal parts of \eqref{comparison}, one gets that $\psi_{u}$ is an eigenvector for $A_{N}$ and satisfies
\begin{equation}\label{weing1}
    A_{N}\psi_{u}=-\frac{f'(\alpha(u))}{f(\alpha(u))}\theta \psi_{u}.
\end{equation}
Therefore, from \eqref{weing2} and \eqref{weing1}, we conclude that $\{\psi_{u}, \psi_{v_{1}},\dots, \psi_{v_{n-1}}\}$ form an orthogonal basis of $A_{N}$ and its expression on that basis takes the form
\begin{align}\label{Wei}
  A_{N} &=\begin{bmatrix}
    -\dfrac{f'(\alpha(u))}{f(\alpha(u))}\theta & 0 & \dots & 0\\
    0 & \dfrac{\sqrt{1-\theta^2}}{\sigma}-\dfrac{f'(\alpha(u))}{f(\alpha(u))}\theta & \dots & 0 \\
    \vdots & \vdots & \ddots & \vdots \\
     0 & 0 & \dots & \dfrac{\sqrt{1-\theta^2}}{\sigma}-\dfrac{f'(\alpha(u))}{f(\alpha(u))}\theta
  \end{bmatrix}.
\end{align}

Now, since we are suppose that $(\Sigma^{n},g,h,\lambda)$ is a gradient almost Yamabe soliton, we obtain from Lemma \ref{prop} that 
\begin{equation}\label{sse}
    (scal_{g}-\lambda)g(X,Y)=\frac{f'(h)}{f(h)}\left[g(X,Y)-dh\otimes dh(X,Y)\right]+\theta g(AX,Y),\qquad \forall X,Y
    \in \mathfrak{X}(\Sigma).
\end{equation}

Notice that, in particular cases $X=\psi_{u}$, $Y=\psi_{v_{i}}$ and $X=\psi_{v_{i}}$, $Y=\psi_{v_{j}}$, $i\neq j$, the orthogonality of $X$, $Y$ and the expression for the height function 
\begin{equation}\label{height}h(u,v_{1},\dots,v_{n})=(\pi_{\mathbb{R}})|_{\Sigma^n}(u,v_{1},\dots,v_{n})=u\sqrt{1-\theta^2}+c_{1},\quad c_{1}\in\mathbb{R},
\end{equation}
implies that equation \eqref{sse} is trivially satisfied. Hence, we need to look at equation \eqref{sse} for a pair of fields $X=Y=\psi_{u}$ and $X=Y=\psi_{v_i}$.

For $X=Y=\psi_{u}$, we obtain
\begin{equation*}
\begin{split}
    (scal_{g}-\lambda)g(\psi_{u},\psi_{u})&=\frac{f'(h)}{f(h)}\left[g(\psi_{u},\psi_{u})-dh\otimes dh(\psi_{u},\psi_{u})\right]+\theta g(A_{N}\psi_{u},\psi_{u})\\
    &=\frac{f'(h)}{f(h)}\left[1-(1-\theta^2)\right]-\frac{f'(h)}{f(h)}\theta^2\\
    &=0.
    \end{split}
\end{equation*}
which implies that $scal_{g}=\lambda$.

Now, for $X=Y=\psi_{v_{i}}$, with $1\leq i\leq n-1$, we get
\begin{equation*}
\begin{split}
    (scal_{g}-\lambda)g(\psi_{v_{i}},\psi_{v_{i}})&=\frac{f'(h)}{f(h)}\left[g(\psi_{v_{i}},\psi_{v_{i}})-dh\otimes dh(\psi_{v_{i}},\psi_{v_{i}})\right]+\theta g(A_{N}\psi_{v_{i}},\psi_{v_{i}})\\
    &=\left(\frac{f'(h)}{f(h)}(1-\theta^2)+\frac{\theta\sqrt{1-\theta^2}}{\sigma}\right)g(\psi_{v_{i}},\psi_{v_{i}}).
    \end{split}
\end{equation*}
Hence, since $scal_{g}=\lambda$, we obtain from above that
\begin{equation}\label{sss}
  \frac{f'(h)}{f(h)}(1-\theta^2)+\frac{\theta\sqrt{1-\theta^2}}{\sigma}=0,
\end{equation}
and then, taking into account equation \eqref{sigma} and \eqref{sss}, it easily follows

\begin{equation}
\begin{split}
    \frac{\sigma_{u}}{\sigma}\psi_{v_{i}}&=\frac{1}{\sqrt{1-\theta^2}}\left[\theta A_{N}\psi_{v_{i}}+\frac{f'(h)}{f(h)}\psi_{v_{i}}\right]=0,
\end{split}
\end{equation}
which implies that $\sigma$ is constant. Therefore, from 
\begin{equation*}
  \frac{f'(h)}{f(h)}(1-\theta^2)+\frac{\theta\sqrt{1-\theta^2}}{\sigma}=0,
\end{equation*}
we deduce that 
\begin{equation*}
    \frac{f'(h)}{f(h)}=\textup{constant}.
\end{equation*}
And thus, $f(t)=c_{3}e^{c_{5}t}$, $c_{3}$, $c_{4}\in \mathbb{R}$.
Bringing together equations \eqref{sa}, \eqref{height} and the expression for $f$ we obtain the desired result.
\end{proof}

\bibliographystyle{abbrv}

\end{document}